\documentclass[reqno]{amsart}
\usepackage{amsmath}
\usepackage{mathtools}
\usepackage{amsthm}
\usepackage{amssymb}
\usepackage{epsfig}
\usepackage{psfrag}
\usepackage{graphicx}
\usepackage{pgf}
\newcommand{\disp}{\displaystyle}
\newcommand{\Rz}{\mathbb{R}}
\newcommand{\Qz}{\mathbb{Q}}

\newcommand{\Nz}{\mathbb{N}}
\newcommand{\epsi}{\varepsilon}

\newcommand{\dr}{{\rm d} r} 
 
\newcommand{\ove}{\overline}
\newcommand{\haz}{\widehat}
\DeclareMathOperator*{\argmin}{{\rm arg\,min}}

\renewcommand{\sl}{|\partial \phi|}
\newcommand{\gl}{\mathfrak{l}_\phi}
\newcommand{\md}{|u'|}

\newcommand{\lan}{\left\langle}
\newcommand{\ran}{\right\rangle}

\newcommand{\GG}{{\mathcal G}}
\newcommand{\EE}{{\mathcal E}}
\newcommand{\FF}{{\mathcal F}}
\newcommand{\HH}{{\mathcal H}}
\renewcommand{\d}{{\rm d}}

\newtheorem{theorem}{Theorem}[section]
\newtheorem{remark}[theorem]{Remark}
\newtheorem{corollary}[theorem]{Corollary}
\newtheorem{definition}[theorem]{Definition}
\newtheorem{proposition}[theorem]{Proposition}
\newtheorem{lemma}[theorem]{Lemma}

\begin{document}

\title[Brezis--Ekeland--Nayroles principle in metric spaces]{The Brezis--Ekeland--Nayroles principle in metric spaces:
time-continuous setting}
 
\author[P.-C. Aubin--Frankowski]{Pierre--Cyril Aubin--Frankowski}
\address[P.-C. Aubin--Frankowski]{CERMICS, CNRS, ENPC, Institut Polytechnique de Paris, 6 and 8 avenue Blaise-Pascal, Cit\'e Descartes,
Champs-sur-Marne, F-77455, Marne-la-Vall\'ee cedex 2, France.}
\email{pierre-cyril.aubin@enpc.fr}
\urladdr{https://pcaubin.github.io/}

\author[G.~E. Sodini]{Giacomo Enrico Sodini}
\address[G.~E. Sodini]{Institute of Analysis and Scientific Computing, TU Wien,
  Wiedner Hauptstra\ss e 8-10, A-1040, Vienna, Austria.}
\email{giacomo.sodini@tuwien.ac.at}
\urladdr{https://giacomosodini.github.io/}

\author[U. Stefanelli]{Ulisse Stefanelli}
\address[U. Stefanelli]{Faculty of Mathematics, University of Vienna, Oskar-Morgenstern-Platz 1, A-1090 Vienna, Austria.}
\email{ulisse.stefanelli@univie.ac.at}
\urladdr{http://www.mat.univie.ac.at/$\sim$stefanelli}


 \begin{abstract}
  Based on a variational characterization of the local
  slope of a functional, we extend the celebrated
  Brezis--Ekeland--Nayroles null-minimization principle to curves of
  maximal slope  in metric
  spaces. In the time-continuous setting, we check that the ensuing
  global-in-time functionals admit minimizers.
 \end{abstract}

 \maketitle

 \section{Introduction}

 \setcounter{equation}{0}

The Brezis--Ekeland--Nayroles principle \cite{Brezis-Ekeland76b,Brezis-Ekeland76, Nayroles76, Nayroles76b}  is a variational
approach to dissipative evolution systems, characterizing solutions as null-minimizers of suitable
nonnegative functionals on trajectories. To introduce the principle, consider for simplicity a doubly nonlinear flow
\begin{equation} \partial \hat \psi( u' ) + \partial \phi(u) = 0 \quad \text{a.e. in} \
[0,T]\label{eq:flow}
\end{equation}
where $u:[0,T]\to V$ is a trajectory in a reflexive Banach space $V$,
the prime denotes  differentiation,  and $\phi, \,\hat \psi: V \to
[0,\infty)$ are convex functionals with single-valued subdifferentials
$\partial \phi, \, \partial \hat \psi:V \to V^*$ (dual).   By introducing a second
trajectory $a:[0,T] \to V^*$, equation
\eqref{eq:flow} can be equivalently written as the system  $-a =
\partial \hat\psi(u')$ and $a =\partial
\phi(u) $ almost everywhere in $[0,T]$.

One can then characterize $\partial \phi: V \to
V^*$ of $\phi$ at a point $u $ with
                                $\phi(u)<\infty$ as
$a \in \partial \phi(u) $ iff $-\lan a , u
\ran + \phi(u) + \phi^*(a)\leq 0$
where $ \phi^*(a) = \sup_{v\in V}\lan a, v \ran -\phi(v)$
is the classical {\it Fenchel conjugate} and $\langle\cdot,\cdot\rangle$ is the duality pairing between
$V^*$ and $V$. In the Banach case one hence
has
\begin{equation}\label{eq:Banach_Fenchel_charac}
    a \in \partial \phi(u)  \ \Leftrightarrow \ -\lan a , u-v
\ran + \phi(u) -\phi(v)\leq 0 \quad \forall v \in V.
\end{equation}
Using this characterization, the Brezis--Ekeland-Nayroles principle for doubly nonlinear flows
\cite[Thm.~1.2]{be} asserts that solutions $u$ to \eqref{eq:flow} correspond
to  null-minimizers  of
\begin{align*}
    \HH_{\rm Banach}(a,u) &=
       \left(\phi(u(T))-\phi(u(0)) +\int_0^T \hat \psi( u') \, \dr +
                             \int_0^T  \hat \psi^*(a)\, \dr \right)^+ \\
  &\quad  +  \int_0^T \sup_{v\in V}\big( - 
                            \langle a,u-v \rangle +
    \phi(u) -\phi(v)\big)\, \dr, 
\end{align*}
denoting $x^+ =x\vee 0$. The  second term in $\HH_{\rm Banach}$ is $0$ iff $a =\partial
\phi(u)$ and, in this case, the first term is $0$ iff   $-a =
\partial \hat\psi(u')$.

We are interested in extending the Brezis--Ekeland--Nayroles
principle to evolution equations in metric spaces. Originating from the pioneering
 contributions \cite{JKO,Otto01},  flows in metric spaces have
 specifically received
 a strong development, motivated by their wide range of possible
 applications. The fundamental reference in the field is the monograph
 by {\sc Ambrosio, Gigli, \& Savar\'e} \cite{Ambrosio08}. In a metric space, however, no
notion of derivative  or duality  is available
and one has to
resort to {\it scalar} fields: the {\it metric
  derivative} $|u'|$ of $u$ and $u \mapsto \sl(u)$ the {\it local slope} of $\phi$.

 A  prominent  notion of dissipative evolution in the complete metric
 space $(U,d)$ is that of {\it curve of maximal slope}. This is
  a  trajectory $u:[0,T]\to U$ solving 
\begin{equation}
  \label{eq:intro_EDE}
  \phi(u(t)) - \phi(u(0)) + \int_0^t\psi\big(|u'|(r)\big) \, \dr +
  \int_0^t\psi^*\big(\sl(u(r))\big)\, \dr=0 \quad \forall t
  \in [0,T]
\end{equation}
where the functional $\phi:U \to [0,\infty]$ and the strictly convex,
smooth  function \linebreak $\psi:[0,\infty) \to [0,\infty)$ with Fenchel conjugate $\psi^*$ are
given.
 
In the
classical {\it geodesically-convex} case  (a setting which we will nonetheless
soon abandon)
  our extension of the classical Brezis--Ekeland--Nayroles principle  is 
based on the global variational characterization of the local slope (which in this case coincides with the global slope) given by
\begin{equation} \alpha = \sl(u) \ \Leftrightarrow \  \alpha =\argmin \big\{\ove
  \alpha \geq 0 :  -\ove \alpha d(u,v){+}\phi(u){-}\phi(v) \leq 0\  \forall v \in
  U\big\},\label{eq:introchar}
\end{equation}
see Proposition
\ref{prop:char} below.
By letting $$ E(\alpha,u) = \max_{v\in U} \big\{- \alpha
d(u,v)+\phi(u)-\phi(v) \big\}$$ (see Lemma \ref{lemma:v} below),  we
define the analogue functional $\HH$ to $ \HH_{\rm Banach}$ over entire trajectories $u:[0,T]\to
U$ and $\alpha: [0,T]\to [0,\infty)$ as
\begin{align*}
    \HH(\alpha,u) &=
       \left(\phi(u(T))-\phi(u(0)) +\int_0^T \psi\big( |u'|^p(r)\big)\,\dr + \int_0^T  \psi^*\big(\alpha(r)\big)\,\dr\right)^+ \\
  &\quad  +  \int_0^T E
                            (\alpha(r),u(r))\, \dr.
\end{align*}
The {\it metric Brezis--Ekeland--Nayroles principle}
ensues by checking in Theorem \ref{prop:BEN} that 
\begin{equation}
  \boxed{\text{$u$ is a curve of maximal slope} \ \Leftrightarrow  \HH(\alpha ,u)
  = \min\HH=0}
  \label{eq:princi}
  \end{equation}
 for some $\alpha$. In this case, $\alpha =\sl(u)$ almost
  everywhere in $[0,T]$. This provides a variational characterization
  of curves of maximal slope, opening the way to applying the methods
  of the calculus of variations  to the evolution problem.  In particular,
the vectorial variable $a(t)\in V^*$ in $\HH_{\rm Banach}$ is replaced by the scalar one
$\alpha(t)\geq 0$ in $\HH$ and the duality $ -\langle a,u-v
\rangle$ is replaced by  $-\alpha d(u,v)$. As  an  effect of this generalization, the
analysis of $\HH$ is new even in the Banach space case.
  
  In
  this paper, we  analyze the above approach in the time-continuous
  setting.   After checking 
  characterization \eqref{eq:princi} (Theorem \ref{prop:BEN}),
we prove that $\HH$ is
  lower semicontinuous (Proposition \ref{prop:FFlsc}) and that it
  admits minimizers (Proposition \ref{prop:min}).  In this paper, we do not explicitly prove that these minimizers are in fact null-minimizers, i.e., $\min_{K(u^0)}\HH=0$. This is however true, as curves of maximal slope do exist, see also \cite{cuore3}.

Besides $\HH$, other functionals may be used to obtain alternative metric
Brezis--Ekeland--Nayroles principles. In this paper, we also
investigate  the
functional  $\GG$ given by 
\begin{align*}
     \GG(\alpha,u) &=
     \left\{
     \begin{array}{l}
      \phi(u(T))-\phi(u(0)) + \int_0^T 
       \psi\big(|u'|(r)\big) \dr +  \int_0^T  \psi^*\big(\alpha
       (r)\big)\, \dr\\[1mm]
       \qquad \text{if} \ \disp \int_0^T E
                            (\alpha(r),u(r))\, \dr=0,\\[3mm]
       \infty\quad \text{otherwise},
     \end{array}
  \right.
\end{align*}
and prove that the characterization \eqref{eq:princi} holds for $\GG$
 as  well (Theorem \ref{prop:BEN}), $\GG$ is lower semicontinuous (Proposition \ref{prop:FFlsc}), and that it
  admits minimizers (Proposition \ref{prop:min}). Moreover, we also
  consider a penalized version of $\GG$. For all $\delta >0$, we let
  the approximating functional $\GG_\delta$ be given by 
  \begin{align*}
    \GG_\delta(\alpha,u) &=\phi(u(T))-\phi(u(0)) + \int_0^T 
       \psi\big(|u'|(r)\big) \dr +  \int_0^T  \psi^*\big(\alpha
       (r)\big)\, \dr \\
                                             &\quad +\frac{1}{\delta}\int_0^T E
                            (\alpha(r),u(r))\, \dr,
                                \end{align*}
                                 we
                                show that minimizers of $\GG_\delta$ converge as $\delta \to 0$ to
a curve of maximal slope up to subsequences (Proposition \ref{prop:pen2}).\\ 
 
\textbf{Related works.} Originally introduced in the linear parabolic
setting, the Brezis--Ekeland--Nayroles approach has been extended to cover gradient flows
\cite{Ghoussoub-Tzou04} and their quadratic perturbations
\cite{Ghoussoub-McCann04}, long-time dynamics \cite{Lemaire96},
compactness \cite{Visintin22}, and
time discretizations \cite{be2}. Extensions to general
monotone and semimonotone flows  \cite{Visintin13,Visintin16},
doubly nonlinear flows  \cite{be,Visintin11, Visintin21}, rate-independent flows
\cite{Visintin13b}, Hamiltonian flows \cite{Buliga, Ghoussoub, Ghoussoub2,
  Ghoussoub3},  and second-order in time problems  \cite{Mabrouk01,Mabrouk03}  are
also available. A far-reaching recast of the
Brezis--Ekeland--Nayroles principle in terms  of   (anti-)selfdual
Lagrangians can be found in \cite{Ghoussoub08}. Applications have been developed to stochastic PDEs
\cite{Barbu,Boroushaki}, optimal control \cite{Barbu12,portinale}, nonlinear
diffusion \cite{Marinoschi,Poliakovsky}, elastoplasticity
\cite{Cao,plas}, conservation laws \cite{Poliakovsky2}, identification
\cite{Barbu-Kunisch95}, and deep learning
\cite{Carini}, among many others.

Beside  the variants of the Brezis--Ekeland--Nayroles
principle mentioned above, global-in-time variational
approaches to dissipative evolution problems have  a long
tradition, dating back at least to Biot's work on irreversible
thermodynamics \cite{Biot55} and Gurtin's principle for
viscoelasticity and elastodynamics
\cite{Gurtin63,Gurtin64b}, see also \cite{Hlavacek69}. In \cite{Visintin01}, solutions to
doubly nonlinear problems are obtained as minimal elements of a
certain partial-order relation on the trajectories. Related to the
specific form  of  the first term in $\HH$, one should also mention the
{\it Energy-Dissipation Principle}, which has attracted
attention  as  an effective tool to identify limits and check
stability, see
\cite{Dondl19,Frenzel,Mielke20,Mielke21,Sandier04,Serfaty11} among
many others.
An alternative global variational approach is the so-called {\it
  Weighted Inertia-Dissipation-Energy} principle, where
solutions are obtained as limits of subsequences of minimizers of
parametrized functionals over trajectories
\cite{Stefanelli25}.
In the setting of gradient flows, the WIDE principle has already been extended to the case of metric spaces in \cite{Rossi11,Rossi19}.\\

 In the second
paper  in  this series  
  \cite{cuore3}, we focus on {\it minimizing-movements}, a reference method to construct curves of maximal slope through 
time discretization \cite{Ambrosio95}. Therein we introduce new schemes  that arise as    {\it variational integrators}
 of time-continuous functionals $\HH$, $\GG$, and $\GG_\delta$. 
  These schemes
 are  shown  to  admit  solutions, whose
time-interpolants converge up to subsequences  to
curves of maximal slope. General conditional convergence results are
also deduced, applying in particular to the case of incremental
quasiminimizers.  In fact, these schemes may serve as a-posteriori convergence  tests, independently  of 
the method used to produce the discrete solutions.  

 The paper is organized as follows. We specify our
  assumptions in Section \ref{sec:setting}. The variational
  characterization of the local slope \eqref{eq:introchar} is then
  discussed in Section \ref{sec:var}. We present the metric
  Brezis--Ekeland--Nayroles principle in Section
  \ref{sec:BEN} and prove the existence of minimizers in Section
  \ref{sec:minimization}.  Finally,    in Section
  \ref{sec:beyond}  we present a further
  generalization of the metric Brezis--Ekeland--Nayroles principle
  under even weaker assumptions on the driving functional $\phi$.

 \section{Setting}\label{sec:setting}
In this section, we collect assumptions \eqref{eq:A1}-\eqref{eq:A4} and fix  the  notation.  To
start with, we assume that 
 \begin{align}
   &(U,d) \ \text{is a complete, separable, path-connected metric
     space, and} \  T \in (0,\infty). \label{eq:A1}
 \end{align}
 The driving functional  $\phi$  is  requested  to be 
  \begin{align}
   &\quad\phi: U \to (-\infty,\infty] \  \text{proper and with compact
     sublevels.}\label{eq:A2}
  \end{align}
A consequence of the compactness of the sublevels of $\phi$ is
 that $\phi$ is lower semicontinuous and bounded from below.
In particular, \eqref{eq:A2} implies that the  {\it effective domain} $D(\phi)\coloneqq\{u \in U :\phi(u)
<\infty \}$ is nonempty and that $\phi$ admits a global minimizer. With
                                no loss of generality, we henceforth 
                                assume that
\begin{equation}\label{eq:A_min0}
    \min \phi=0.
\end{equation}
For given $p \in (1,\infty)$,   following  the notation in \cite{Ambrosio08}, the set $AC^p([0,T];U)$
indicates {\it $p$-absolutely continuous} curves, namely, trajectories 
$u: [0,T]\to U$ such that there exists $\eta\in
L^p(0,T)$ with 
\begin{equation*}
d(u(s),u(t))\leq \int_s^t \eta(r)\,\dr \quad\forall 0 \leq
s\leq t < T.
\end{equation*}
If $u
\in AC^p([0,T];U) $, the limit
$$|u'|(t) \coloneqq \lim_{s\to t}\frac{d(u(s),u(t))}{|t-s|}$$
exists for almost  all  $t\in (0,T)$, see
\cite[Thm. 1.1.2, p.~24]{Ambrosio08}, and is called  the  {\it metric
derivative} of $u$ at $t$. In this case, the map
 $t \mapsto|u'|(t)$  is in $ L^p(0,T) $
and coincides with the minimal function $\eta\in L^p(0,T)$ satisfying the above bound.

The {\it local slope} $\sl(u)$ and the {\it global slope} $\gl(u)$ \cite{Ambrosio08,Cheeger99,DeGiorgi80} of $ \phi
$ at $ u \in U$ are defined via
$$
\sl(u) \coloneqq \limsup_{v \to u} \frac{(\phi(u) -
  \phi(v))^+}{d(u,v)}, \quad \gl(u) \coloneqq \sup_{v \not =u}\frac{(\phi(u) -
  \phi(v))^+}{d(u,v)},
$$
respectively, and we indicate the corresponding domains as  
$D(\sl) \coloneqq \{u \in D(\phi): \sl(u)<\infty\}$ and $D(\gl)  \coloneqq \{u \in D(\phi): \gl(u)<\infty\}$. Note that the definition
of $\sl(u)$ implicitly requires that $u$ is not isolated   which actually cannot occur as $U$ is  
path-connected by \eqref{eq:A2}.

 Being defined via a supremum, 
the global slope $\gl$ is lower semicontinuous.  On the other hand, the local slope $\sl$  is defined via a $\limsup$ and is not, in general, lower semicontinuous. In order to  approximate $\sl$, we  introduce  a penalized formulation of $\gl$ through a $\mu \geq 0$. For some given $q>1$ and for all $\mu \geq 0$, we define the {\it $\mu$-slope} $\sl_\mu(u)$ of $\phi$ at $u$ as 
\begin{equation}\label{eq:slmu}
\sl_\mu(u) \coloneqq \sup_{v\not = u}\left(\frac{(\phi(u) -
    \phi(v))^+}{d(u,v)} - \mu d^{q-1}(u,v) \right).  
\end{equation} 
Note that $\sl_\mu$ is lower semicontinuous \cite[Prop. 2.7]{rsss} and  that
$\sl_0 = \gl$.

We say that a function $g: U \to [0,\infty]$ is a {\it strong upper gradient} for $\phi$
\cite[Def.~1.2.1, p.~27]{Ambrosio08} if for all $u\in AC^p([0,T];U)$,
the map
$t \mapsto g(u(t))$ is Borel and 
$$|\phi(u(t)) - \phi(u(s))| \leq \int_s^t g (u(r))\,\md(r) \, \dr
\quad \forall 0\leq s \leq t \leq T.$$
If $t \mapsto g (u(t))\md(t) \in L^1(0,T) $ the
latter entails that $\phi \circ u  \in W^{1,1}(0,T) $ and $(\phi
\circ u)' \le g (u) \md$ almost everywhere in $[0,T]$.  Note that
 $\sl_\mu$  is a strong upper gradient for all $\mu \in [0,\infty)$ based on  \cite[proof of Corollary 2.4.10]{Ambrosio08}. In
particular, 
$\gl$ is a strong upper gradient.

In what follows, we will assume that
\begin{align}
    \exists
  \mu_0 \in [0,\infty) \ \text{such that} \ \sl=\sl_{\mu_0}. \label{eq:A20}
\end{align}
 The existence of such $\mu_0$ entails in particular that $\sl$
is lower semicontinuous and a strong upper gradient for $\phi$. Assumption
\eqref{eq:A20} in particular holds if $\phi$ is {\it
  $(\lambda,q)$-generalized-geodesically convex} for some $\lambda \in
\Rz$  \cite{rsss}  (see also \cite{Ohta-Zhao,Shimoyama} for
alternative but comparable definitions).  In this case, one can choose   $\mu_0 = \lambda^-/q$, where
$\lambda = -\lambda \vee 0$. In Section \ref{sec:beyond}, we drop this assumption \eqref{eq:A20}. 

The initial value $u^0$ is asked to fulfill
\begin{equation}
  u^0\in D(\phi).\label{eq:A3}  
\end{equation}

The dissipative character of the evolution is quantified via  the function $\psi:\Rz \to [0,\infty)$ satisfying 
\begin{align}
  &\psi\in C^1(\Rz) \ \text{convex  with}\nonumber\\[2mm]
  &\text{$\psi(0) = \min\psi =0$ such that there exist} \  p\in(1,\infty)  \ \text{and} \ c_\psi\in(0,\infty) \
    \text{with}  \nonumber\\
  &\psi(t)+\psi^*(t^*) \geq c_\psi |t|^p+c_\psi |t^*|^{p'}-\frac{1}{c_\psi}\quad \forall t,\, t_*
 \in \Rz
  \label{eq:A4}
\end{align}
Here, $1/p+1/p'=1$ and $\psi^*$ is the classical Fenchel conjugate $\psi^*(t^*)= \sup_t(
t^*t-\psi(t))$ for $t^*\in \Rz$.  The reference example is
$\psi(t) =|t|^p/p$.  By duality one has that 
\begin{equation} \psi(t)+\psi^*( t^*) \leq C_\psi(1+ |t|^p+|t^*|^{p'})\quad \forall t,\, t_*
\in \Rz\label{eq:upperbound0}
\end{equation}
for some $C_\psi\in(0,\infty)$ depending on $c_\psi$ and $p$.  Indeed, by \eqref{eq:A4} evaluated first at $t^*=0$ and then at $t=0$, we have
\begin{align*}
    \psi(t) \ge c_\psi |t|^p - \frac{1}{c_\psi}, \quad 
    \psi^*(t^*) \ge c_\psi |t^*|^{p'} - \frac{1}{c_\psi} \quad \forall t,t^* \in \Rz,
\end{align*}
so that
\begin{align*}
    \psi(t) &= \psi^{**}(t) \le \sup_{t^*} \{tt^* -c_\psi |t^*|^{p'}\} + \frac{1}{c_\psi} = \frac{1}{p(c_\psi p')^{p-1}} |t|^p + \frac{1}{c_\psi} \quad \forall t \in \Rz,
    \\
    \psi^*(t^*) & \le \sup_{ t } \{tt^* -c_\psi |t|^{p}\} + \frac{1}{c_\psi} = \frac{1}{p'(c_\psi p)^{p'-1}} |t|^{p'} + \frac{1}{c_\psi} \quad \forall t^* \in \Rz,
\end{align*}
which give \eqref{eq:upperbound0} with $C_\psi\coloneqq \max\{ 2c_\psi^{-1}, (p'(c_\psi p)^{p'-1})^{-1}, (p(c_\psi p')^{p-1})^{-1}\}$.  Note that  the
convexity of $\psi$ and the fact that $0=\psi(0)=\min \psi$ imply that
$\psi$ is nondecreasing on $[0,\infty)$. On the other hand, the
differentiability of $\psi$ ensures that $\psi^*$ is strictly convex,
which, together with $0=\psi^*(0)=\min\psi^*$, gives that $\psi^*$ is
strictly increasing on $[0,\infty)$.

 Assumptions \eqref{eq:A1}--\eqref{eq:A2} and
\eqref{eq:A20}--\eqref{eq:A4} are tacitly assumed  throughout   the paper,
with the exception of Section \ref{sec:beyond} where only the assumption \eqref{eq:A20} is dropped.

In this paper, we focus on the following classical notion  \cite{Muratori-Savare} in studying the evolution of gradient flows. 

\begin{definition}[Curve of maximal slope]\label{def:curve}
  The trajectory $u\in AC^p([0,T];U)$ is said to be a \emph{curve of
    maximal slope}
  if, $u(0)=u^0$, $\phi\circ u \in W^{1,1}(0,T)$, and 
  \begin{equation}
    \label{eq:curve}
    \phi(u(t)) + \int_0^t \psi(|u'|(r)) \, \dr  +   \int_0^t
    \psi^*(\sl(u(r)))\, \dr =
    \phi(u(0)) \quad \forall t \in [0,T].
  \end{equation}
\end{definition}

The latter is a slight generalization of the concept of {\it $(p,q)$-curves
of maximal slope} from \cite[Def.~1.3.2, p.~32]{Ambrosio08},
which indeed correspond to the choice $\psi(t)=|t|^p/p$.

To conclude this preliminary section,  let us collect here some compactness tools,
which are used in the following.

\begin{lemma}[Compactness tools] \label{lemma:compactness} Let $ (u_j)_j \subset  AC^p([0,T];U)$ have
  $|u_j'|\to \eta$ weakly in $L^p(0,T)$.  
  \begin{itemize}
  \item[\rm (a)] \ In case that $$\sup_j \sup_{t \in [0,T]} \phi(u_j(t))<\infty,$$
    there  exist $u \in AC^p([0,T];U)$  with $|u'| \leq \eta$ almost everywhere in
    $[0,T]$  and  a not relabeled sequence such that $u_j(t)\to u(t)$
    for all $t\in [0,T]$.
     \item[\rm (b)] \ In case that, for some $t_0 \in [0,T]$,
       $(u_j(t_0))_j$ is  precompact  and  $$\sup_j\int_0^T \phi(u_j(t))\, \d t<\infty,$$
     there  exist $u \in AC^p([0,T];U)$  with $|u'| \leq \eta$ almost everywhere in
    $[0,T]$   and  a not relabeled sequence such that $u_j(t)\to u(t)$
    for almost all $t\in [0,T]$  and $u_j(t_0) \to u(t_0)$.
  \end{itemize}
\end{lemma}
 The results (a) and (b) above correspond to
  the Ascoli--Arzel\`a theorem  and  a metric version of the Aubin--Lions lemma,   and are proved in
\cite[Prop. 3.3.1]{Ambrosio08} and \cite[Thm.~2]{Rossi-Savare}. 
The only difference between the result (b) above and the one presented
in \cite[Thm.~2]{Rossi-Savare} is the convergence at the additional
point $t_0$, which we briefly comment here. Since $(u_j(t_0))_j$ is
assumed to be  precompact,  we can first extract a not
relabeled subsequence such that $u_j(t_0) \to  \ove u \in U$. 
On the other hand, the result by \cite[Thm.~2]{Rossi-Savare} gives the
existence of $u \in AC^p([0,T];U)$   with $|u'| \leq \eta$ almost
everywhere in $[0,T]$  and  of  a not relabeled sequence such that $u_j(t)\to u(t)$ for almost all $t\in [0,T]$. In particular, we have
\begin{align*}
d(u_j(t_0), u(t_0)) &\le d(u_j(t_0), u_j(t)) + d(u_j(t), u(t))+d(u(t), u(t_0)),
\\
& \le \int_{t_0 \wedge t}^{t_0 \vee t} |u'_j|(r) \d r + d(u_j(t), u(t)) + d(u(t), u(t_0))
\end{align*}
for every $t \in [0,T]$ such that $u_j(t) \to u(t)$. Passing to the limit as $j \to \infty$, we get
\[
d( \ove u  , u(t_0))  \leq   \int_{t_0 \wedge t}^{t_0 \vee t} \eta(r) \d r + d(u(t), u(t_0))
\]
and passing to the limit as $t \to t_0$, we obtain $u(t_0)= \ove u$.

\section{Variational characterization of $\sl(u)$}\label{sec:var}

This section pertains to some introductory discussion concerning
the variational characterization of the slope $\sl$.
Let us start by defining the functional $\ove E: [0,\infty) \times U \times
D(\phi) \to (-\infty,\infty]$ as
\begin{equation*}
    \ove E(\alpha,u,v) \coloneqq -\alpha d(u,v) + \phi(u) -\phi(v)- \mu_0
 d^q(u,v).
\end{equation*}
We  first  observe the following.
\begin{lemma}[Reduced functional]\label{lemma:v}
  For all $(\alpha,u)\in [0,\infty)\times U$  there exists the maximum   \begin{equation}
  E(\alpha,u) \coloneqq   \max_{ v\in D(\phi)}   \ove
 E(\alpha,u,v) \label{eq:argmax}
\end{equation}
 which  defines the reduced functional $E : [0,\infty)\times U
 \to [0,\infty]$. 
\end{lemma}
\begin{proof}   If $u \notin D(\phi)$, then $\ove E(\alpha, u, v)=\infty$ for every $(\alpha, v) \in [0,\infty) \times D(\phi)$, so that the  maximum  in \eqref{eq:argmax} is $\infty$ and  is  achieved at any $v \in D(\phi)$.  Assume now that   $u \in D(\phi)$.
  As $\phi$ is
  lower semicontinuous,  for all $(\alpha,u)   \in [0,\infty) \times D(\phi)  $  the map $ v \in U
  \mapsto  \ove E(\alpha,u,v) = -\alpha d(u, v)+\phi(u)
  -\phi( v)  - \mu_0 d^q(u, v) \in [-\infty,\infty)$ is upper
  semicontinuous. Moreover, it is 
  finite iff $  v \in D(\phi)$. Letting $u_*\in \argmin \phi$ and using $\phi(u_*)=0$ by \eqref{eq:A_min0}, we can reduce the
  maximization  in \eqref{eq:argmax}  to those $ v\in D(\phi)$ such that $$-\alpha d(u, v)
  -\phi( v) - \mu_0  d^q(u, v)\geq -\alpha d(u,u_*)
  - \mu_0  d^q(u,u_*).$$
 Hence, one can maximize 
    over
  the compact set $\{ v \in D(\phi):\phi( v) \leq \alpha
  d(u,u_*) +{ \mu_0  }d^q(u,u_*)  \}$. The existence of a maximizer  follows by the Direct Method.
\end{proof}

Note that we are not claiming uniqueness for problem
\eqref{eq:argmax}. This nonetheless follows if $v \mapsto \alpha
d(u,v)+\phi(v) +  \mu_0  d^q(u,v)$ is {\it strictly} geodesically
convex for all $(\alpha,u)$. This holds independently of $\alpha
\geq 0$ if $\phi$ is
strictly geodesically convex and both $v\mapsto d(u,v)$ and $v\mapsto
 \mu_0  d^q(u,v)$ are 
convex along any geodesic. 

  Note that $ E(\alpha,u) \geq \ove E(\alpha,u,u)=0$  for all $\alpha\geq 0$ and
  $u \in  D(\phi)$. Moreover, $ E$ is convex with respect to $\alpha$, being the supremum of linear
  functions in $\alpha$. We have $ E(\alpha,u)=\infty$ for all $\alpha\geq
  0$ iff $u \not \in D(\phi)$. We collect some properties of $ E$ in the following.

\begin{proposition}[Properties of $ E$]\label{lemma:lsc}  
The reduced functional $ E:[0,\infty) \times U \to [0,\infty]$ is lower semicontinuous and there exists a Borel map  $\haz
v: [0,\infty) \times U \to U$ such that $ E(\alpha,u) =
\ove E(\alpha,u,\haz v(\alpha,u))$ for all $(\alpha,u)\in  [0,\infty) \times U $.
\end{proposition}

\begin{proof} The lower semicontinuity is immediate as the maps
  $(\alpha,u)\in [0,\infty) \times U \mapsto \ove E(\alpha,u,v)$ are lower
  semicontinuous for all $ v \in D(\phi)$ and $ E(\alpha,u)$
  is defined as a maximum on $v$. This also implies that $ E$ is a
  Borel map.

  Let $\Gamma \subset [0,\infty) \times U \times U$ be defined as
  \[
  \Gamma\coloneqq \{ (\alpha, u,v) \in [0,\infty) \times U \times D(\phi) :
  \ove E(\alpha, u, v) =  E(\alpha, u) \}.
  \]
  Since $\ove E$ and $ E$ are Borel maps, then $\Gamma$ is a Borel
  set. Its sections $\Gamma_{\alpha, u}$, given, for every $(\alpha, u) \in [0,\infty)\times U$, by
  \[
  \Gamma_{\alpha, u} \coloneqq \{ v \in U : (\alpha, u,v) \in \Gamma \} = \{
  v \in D(\phi) : \ove E(\alpha, u,v) =  E(\alpha, u)\},
  \]
  are not empty by Lemma \ref{lemma:v} and compact by  the 
  upper semicontinuity of $D(\phi) \ni v \mapsto \ove E(\alpha, u,v)$ and
   by   the compactness of the sublevels of $\phi$. By
  \cite[Theorem 6.9.6]{Bog07} we  obtain  the existence of a Borel map $\haz v: [0,\infty) \times U \to U$ such that its graph is contained in $\Gamma$.
\end{proof} 

 Under assumption \eqref{eq:A20},  for all $u \in
    U$   one has that 
    \begin{align} &\sl(u) \leq \alpha\  \Leftrightarrow \   \ove
    E(\alpha,u,v) =  - \alpha d(u,v)+\phi(u)-\phi(v) -{ \mu_0 
        } d^q(u, v)\leq 0\quad \forall v \in
  U.\label{eq:compare}
\end{align}
As $ E(\alpha,u)\geq 0$ for all $u \in D(\phi)$, the latter yields 
  \begin{equation} \sl(u) \leq
  \alpha \ \Leftrightarrow \  E(\alpha,u)=0. \label{eq:max}
\end{equation}
  This implies the following variational
characterization, which corresponds to a metric version of the
{\it Fenchel identity} in Banach spaces.

 \begin{proposition}[Characterization]\label{prop:char}
  Given $u\in D(\phi)$, we have that
   $$u \in D(\sl) \ \Leftrightarrow  \  A_u =\{ \alpha \geq 0 : 
   E( \alpha,u) =0\} \not = \emptyset.$$
In this case, $\sl(u)=\min  A_u $. 
 \end{proposition}
 \begin{proof}
   If $u \in D(\sl)$ one has that $ E(\sl(u),u)=0$ from
   \eqref{eq:max}, hence $\sl(u) \in  A_u \not =\emptyset$.  Conversely,   if
   $ \alpha  \in  A_u $ then $ E( \alpha,u)=0$ so that $\sl (u)\leq
   \alpha<\infty$ and $u \in D(\sl)$. In
   this case, again \eqref{eq:max} guarantees that $\sl(u)$ is minimal
   in $ A_u $. 
 \end{proof}

 Proposition \ref{prop:char} ensures that for all $u \in D(\phi)$ one can variationally
 characterize $\sl(u)$ as follows
 \begin{equation}\sl(u)  = \argmin_{ E(
   \alpha,u)=0}   \psi^*( \alpha) ,\label{eq:strong}
 \end{equation}
 where we recall that $\psi^*$ from \eqref{eq:A4} is strictly
 increasing on $[0,\infty)$. As $ E$ is nonnegative, this provides the possibility of penalizing the equation
 $\alpha=\sl(u)$ by considering an unconstrained minimization problem.

 \begin{proposition}[Penalization]\label{prop:pen}
     For all $u \in D(\phi)$ and $\delta >0$
   one can find
   \begin{equation}\label{eq:min}
     \alpha_\delta \in \argmin_{\alpha\geq 0}  \left\{\psi^*( \alpha) + \frac1\delta  E(
     \alpha,u) \right\}.
   \end{equation}
   If $u \in D(\sl)$,  we have that $\alpha_\delta \to \sl(u)$ as
   $\delta \to 0$. 
 \end{proposition}

 \begin{proof}
Let us first check that the minimization problem \eqref{eq:min}  can be solved.
   As $ E$ is nonnegative  and finite on $D(\phi)$  and $\psi^*$ is coercive, one has
   that $ \alpha \in [0,\infty) \mapsto \psi^*( \alpha) +  E(
   \alpha,u)/\delta$ is coercive, as well. The existence of $\alpha_\delta$ for
   all $\delta >0$
   follows then by the lower semicontinuity of $\psi^*$ and by
   Proposition \ref{lemma:lsc}.

   Assume now that $u \in D(\sl)$. By comparing with $\sl(u)$ and
   using the fact that $ E(\sl (u),u)=0$
   we obtain that 
   \begin{equation}\psi^*( \alpha_\delta) + \frac1\delta  E(
   \alpha_\delta,u) \leq \psi^*(\sl(u)) +\frac1\delta  E(\sl (u),u)   =
   \psi^*(\sl(u)).\label{eq:estimf}
   \end{equation}
   On the one hand, as $ E$ is nonnegative this proves that $\alpha_\delta$ is bounded
   independently of $\delta$, so that $\alpha_\delta \to \alpha$ (up to
   some not relabeled subsequence) as
   $\delta \to 0$. On the other hand, as $\psi^*$ is nonnegative, inequality \eqref{eq:estimf} implies that
   $$ 0\leq  E(\alpha,u) \leq \liminf_{\delta \to 0}
   E(\alpha_\delta, u) \leq \liminf_{\delta \to 0} \delta\, \psi^* (\sl(u))
   = 0$$
   proving that $ E(\alpha,u)=0$, or $\sl(u)\leq \alpha$ by \eqref{eq:max}. Hence, we get
   that $\psi^*(\sl(u))\leq \psi^*(\alpha)$ since $\psi^*$ is
   (strictly) increasing on $[0,\infty)$. Again by
   \eqref{eq:estimf}, we get
$$\psi^*(\alpha) \leq \liminf_{\delta \to 0}\psi^*(\alpha_\delta)
\leq \psi^*(\sl(u))$$
   which implies $\alpha = \sl(u)$, as $\psi^*$ is strictly
   increasing on $[0,\infty)$. The above argument  ensures  that any subsequence of
   $(\alpha_\delta)_\delta$ has a converging subsequence to the same
   limit. This proves that the whole sequence $(\alpha_\delta)_\delta$
   converges to
   $\sl (u)$ and
    no extraction is actually necessary.
  \end{proof}

\section{Metric  Brezis--Ekeland--Nayroles principle}
\label{sec:BEN}

 In this section, we introduce the metric version of the
Brezis--Ekeland--Nayroles principle (Theorem \ref{prop:BEN}) and discuss some of its
properties. Let us   endow $[0,T]$ with the Lebesgue measure and indicate
  by $M([0,T]; U)$ and $M([0,T];[0,\infty)\times U)$  the sets of Borel
  measurable maps with values in $U$ and $[0,\infty) \times U$, respectively. Owing to
  Proposition \ref{lemma:lsc}, for all $(\alpha,u) \in
  M([0,T];[0,\infty)\times U)$ we have that $t \in [0,T]\mapsto  
  E(\alpha(t),u(t))$ is measurable and $ v\coloneqq\haz v(\alpha,u) \in
  M([0,T];U)$ is such that $
  E(\alpha(t),u(t))= \ove E(\alpha(t),u(t), v (t))$ for all $t \in
  [0,T]$. Recalling that $ E $  is nonnegative,  we  define $ \EE: M([0,T];[0,\infty) \times
 U) \to [0,\infty]$ as
  $$ \EE(\alpha,u) = 
  \disp\int_0^T E(\alpha(r),u(r))\, \dr.$$
Analogously to the static setting,  we have that $\EE $  is
nonnegative and  
$\EE(\alpha,u) < \infty$ implies that $u  \in D(\phi)$ almost
everywhere in 
$[0,T]$.

By recalling that $u^0 \in D(\phi)$ from
\eqref{eq:A3}, we indicate by 
\begin{equation}
  \label{eq:K}
  K(u^0)\coloneqq\{(\alpha,u)\in L^{p'}(0,T)\times AC^p([0,T];U) : u(0)=u^0\}
\end{equation}
the set of {\it admissible trajectories}. Moreover, we define the functionals $\FF: L^{p'}(0,T)\times
AC^p([0,T];U) \to (-\infty,\infty]$ and $\GG,\, 
\GG_\delta,  \HH : L^{p'}(0,T)\times
AC^p([0,T];U)\to (-\infty,\infty]$ for $\delta>0$ as
\begin{align*}
   &\FF(\alpha,u) = 
     \phi(u(T))-\phi(u(0)) + \int_0^T\psi(
       |u'|(r))\, \dr +  \int_0^T \psi^*(\alpha (r))\, \dr , \\
   & \GG(\alpha,u) =
     \left\{
     \begin{array}{ll}
       \FF(\alpha,u) \quad&\text{if} \ \EE (\alpha,u)=0,\\[2mm]
       \infty&\text{otherwise},
     \end{array}
               \right. \\
  &\GG_\delta(\alpha,u) = \FF(\alpha,u) + \frac1\delta \EE(\alpha,u),\\[2mm]
  & \HH(\alpha,u) = (\FF(\alpha,u))^+ + \EE(\alpha,u).
\end{align*}

We are now ready to present  our metric version of the Brezis--Ekeland--Nayroles principle \cite{Brezis-Ekeland76b,Brezis-Ekeland76,Nayroles76,Nayroles76b}, delivering a
characterization of
curves of maximal slope. 
As a preparation, let us remark that  both $\HH $ and
 $\GG$  are  nonnegative on
$K(u^0)$.  Indeed, $\HH \geq 0$ trivially follows as $\EE
\geq 0$. Assume  by contradiction that $\GG(\alpha, u)<0$ for some
$(\alpha , u) \in K(u^0)$. As $ \EE ( \alpha,u)=0$, we have  
  that $\sl( u) \leq \alpha $ almost everywhere in $[0,T]$  by \eqref{eq:max}, and
$\FF(\alpha,u)=\GG(\alpha,u)<0$.  As $\psi^*(\sl( u))\leq \psi^*(\alpha)$ almost everywhere in
 $[0,T]$, we have $\FF(\sl( u),  u)\leq \FF(\alpha,  u)<0$. On the other
 hand, as $\sl$ is a strong upper gradient, one finds 
 $\phi\circ u\in W^{1,1}(0,T)$ and the chain rule yields
 \begin{align*}
   &\FF(\sl( u) , u) = \phi( u(T))-\phi(u(0))+ \int_0^T \psi(| u'|(r))\,
   \dr + \int_0^T \psi^*(\sl( u(r)))\, \dr \\
   &\quad \geq \int_0^T \big( - \sl( u(r)) | u'|(r)
                       +\psi(| u'|(r))+\psi^*( \sl( u(r))\big)
                       \dr \geq 0
 \end{align*}
 from Fenchel's inequality, leading to a contradiction.

 \begin{theorem}[Null-minimization]\label{prop:BEN}
    The following are equivalent:
   \begin{itemize}
   \item[\rm (i)] \ \  $u$ is a curve of maximal slope,\vspace{0.6mm}
     \item[\rm (ii)] \  \ $
  \GG(\alpha,u)=\min_{K(u^0)} \GG=0$ for some $\alpha \in
  L^{p'}(0,T)$,
  \item[\rm (iii)]  \ \   $
  \HH(\alpha,u)=\min_{K(u^0)} \HH=0$ for some $\alpha \in
  L^{p'}(0,T)$.
   \end{itemize}  
\end{theorem}
\begin{proof}
  Let $u$ be a curve of maximal slope. Then $u(0)=u^0$, $\sl (u) \in
  L^{p'}(0,T)$, and \eqref{eq:curve} for
  $t=T$ gives $\FF(\sl(u),u)=0$. Since in particular $u \in D(\sl) $
  almost everywhere in $[0,T]$, we have that $\EE(\sl(u),u)=0$ and $
  \GG(\alpha,u)  = \HH (\alpha,u)  =0$ follows by choosing
  $\alpha = \sl(u)$.  Hence (i) implies both (ii) and (iii). 

 We now check (ii) $\Rightarrow$ (i). Let    $(\alpha,u)\in K(u^0)$ with  
  $\GG(\alpha,u)=0$. This implies
  that  $\EE(\alpha,u) = 0$, so that $\FF(\alpha,u) =   \GG(\alpha,u)=0$. As $\sl(u) \leq \alpha$ almost everywhere in $[0,T]$  by \eqref{eq:max}, we have proved that $u \in
  D(\sl)$ almost everywhere in $[0,T]$ and  
  $\sl(u)\in L^{p'}(0,T)$. As $\sl$ is a strong upper gradient for
  $\phi$, this implies that $\phi\circ u\in W^{1,1}(0,T)$. In
  particular, for all $t\in [0,T]$ we have that
  \begin{align*}
    &\phi(u^0)-\phi(u(t))\leq |\phi(u(t))-\phi(u^0)| \leq \int_0^t \sl(u(r))\, |u'|(r)\, \dr \\
    &\quad\leq
    \int_0^t\psi( |u'| (r))\, \dr +  \int_0^t\psi^*(\sl  (u(r)))\,
    \dr 
  \end{align*}
 giving the inequality
  \begin{equation} \phi(u(t)) - \phi(u^0) + 
     \int_0^t\psi( |u'| (r))\, \dr +  \int_0^t\psi^*(\sl (u(r)))\,
    \dr \geq 0.\label{eq:q1}
    \end{equation}
 On the other hand, from $\sl(u) \leq \alpha$ almost everywhere in
 $[0,T]$ we deduce that $\psi^*(\sl(u))\leq \psi^*(\alpha)$ almost everywhere in
 $[0,T]$, hence $\FF(\sl(u),u)\leq \FF(\alpha,u)=0$.  By using the fact that 
    \begin{align}
      &-\phi(u(T)) +\phi(u(t)) \leq |\phi(u(T))-\phi(u(t))|\leq
      \int_t^T \sl(u(r))\, |u'|(r)\, \dr \nonumber \\
      &\quad \leq
   \int_t^T\psi( |u'| (r))\, \dr +  \int_t^T\psi^*(\sl  (u(r)))\,
    \dr, \label{eq:q2}
    \end{align}
    we can check that 
  \begin{align*}
    &\phi(u(t))-\phi(u^0) + \int_0^t\psi( |u'| (r))\, \dr +
      \int_0^t\psi^*(\sl  (u(r)))\,
    \dr  \\[1mm]
    &\quad = \FF(\sl (u),u) -\phi(u(T))+\phi(u(t)) \\
    &\quad - \int_t^T\psi( |u'| (r))\, \dr -
      \int_t^T\psi^*(\sl (u(r)))\,
    \dr   \leq 0.
  \end{align*}
  where in the last inequality we have used $\FF(\sl (u),u)\leq \FF(\alpha,u)=0$ and
  \eqref{eq:q2}. This proves the converse inequality to
  \eqref{eq:q1}. As $t \in [0,T]$ is arbitrary,   \eqref{eq:curve}
  follows and $u$ is a curve of maximal slope.

   Let us now prove that (iii) $\Rightarrow$ (i). As
  $\HH(\alpha,u)=0$ we again have that $  \EE(\alpha,u)=0$, as well as
  $\FF(\alpha,u) \leq 0$. As $  \sl(u)\leq \alpha$ almost everywhere in $[0,T]$, this implies that  
  $\FF(\sl(u),u)\leq \FF(\alpha,u) \leq 0$. We can hence  obtain 
  \eqref{eq:q1}--\eqref{eq:q2}, proving again that $u$ is a curve of
  maximal slope. 
\end{proof}

\begin{remark}[$\GG_\delta$ may be negative]
  Before moving on, let us remark that we cannot expect  Theorem
   \ref{prop:BEN} to hold for $\delta>0$ as $\GG_\delta$ may take
  negative values. Indeed, if $\delta >0$ it might be convenient to
  choose some $\alpha<\sl(u)$ making $\FF(\alpha,u)$ negative, despite
   having  a positive $\EE(\alpha,u)$. Let us provide an
  elementary example in this direction. Set $U=[0,\infty)$,
  $\phi(u) = \lambda u^2/2$ for some $\lambda >0$, and
  $\psi(t) = t^2/2$. In this case, we readily compute that
  $\sl(u) =\lambda u$. For all $\alpha \geq 0$ and $ u \geq 0$ we 
  have  that
  $ E(\alpha, u) = ((\lambda u - \alpha)^+)^2/(2\lambda)$. One then
  considers
$$ [0,\infty] \ni \alpha  \mapsto g_{ u}( \alpha)\coloneqq\frac12 \alpha^2 +
\frac{1}{\delta} E(\alpha, u) = \frac12 \alpha^2 +
\frac{1}{2\delta\lambda }((\lambda u - \alpha)^+)^2$$ which is
minimized at $ \alpha_{ u} \coloneqq \lambda u/ (1+ \lambda \delta)$
with value
$g_{ u}( \alpha_{ u}) = \lambda^2 u^2/(2(1+\lambda \delta)) $. Hence,
for all $ u \in AC^2([0,T];U) = H^1(0,T;U)$ with $ u>0$ we compute
\begin{align*}
  & \GG_\delta( \alpha_{ u}, u) = \FF(\alpha_{
    u}, u)+\frac{1}{\delta}\EE( \alpha_{ u}, u) \\
  &\quad=
    \frac{\lambda}{2}u^2(T) -  \frac{\lambda}{2}u^2(0) +
    \frac12\int_0^T | u'(r)|^2\, \dr +
    \frac{1}{ 2(1+\lambda\delta)}\int_0^T |\lambda  u(r)|^2\,
    \dr\nonumber\\
  &\quad <  \frac{\lambda}{2}u^2(T) -  \frac{\lambda}{2}u^2(0) +
    \frac12\int_0^T | u'(r)|^2\, \dr +
    \frac{1}{2}\int_0^T |\lambda  u(r)|^2\,
    \dr \nonumber\\
  &\quad =\FF(\lambda  u, u)\leq \GG_\delta (\lambda 
    u, u).
\end{align*}
By choosing $u_0=1$, setting
$t\in [0,T]\mapsto \ove u(t) \coloneqq {\rm e}^{-\lambda t}$, and
observing that $\ove u '+\lambda \ove u=0$  on $[0,T]$  and
$(\lambda \ove u ,\ove u)\in K(u^0)$, one has that
$\GG_\delta(\lambda \ove u ,\ove u)=0$. This implies that
$\GG_\delta(\alpha_{\ove u}, \ove u)<\GG_\delta(\lambda \ove u ,\ove
u)=0$.
\end{remark}


\section{Minimization} \label{sec:minimization}
 
 Theorem  \ref{prop:BEN} links curves of maximal slope with
($u$-components of) minimizers of $\GG$  and $\HH$  
on $K(u^0)$. This opens
the way  to   providing an alternative proof of the existence of curves
of maximal slope by minimization. We start by checking lower
semicontinuity.

\begin{proposition}[Lower semicontinuity of $\EE$]\label{prop:BBlsc}
   Let  $\alpha_j \to \alpha$ weakly in $L^{ p'}(0,T)$, $u_j, \, u \in
   M  ([0,T];U)$ with  $u_j\to u$ almost everywhere in
  $[0,T]$,     and 
  
  \[
 d(u_j(s), u_j(t)) \le \int_s^t \eta_j(r) \dr \quad \forall 0 \le s \le t \le T, j \in \Nz
  \]
   for some $\eta_j \in L^1(0,T)$ converging to $\eta$ weakly in $L^1(0,T)$.  Then, 
  \begin{equation*}
    \EE (\alpha,u)\leq \liminf_{j\to \infty} \EE(\alpha_j,u_j).
  \end{equation*} 
\end{proposition}
\begin{proof} 
  Let $ v = \haz v(\alpha,u) \in M([0,T];U)$ from Proposition
  \ref{lemma:lsc}, so that $ E(\alpha,u) =
  \ove E(\alpha,u, v)$  almost everywhere in  $[0,T]$.
 Using the triangle inequality
  we readily check that
$$|d(u(t), v (t)) - d(u_j(t), v(t))| \leq d(u(t),u_j(t))\to 0
\quad  \text{for a.e.} \ t \in [0,T].$$
 
   Let $t_0\in [0,T]$ be such that $u_j(t_0)\to u(t_0)$. For almost all $t \in [0,T]$ we have that
  $$d(u(t),u(t_0)) = \lim_{j\to \infty}d(u_j(t),u_j(t_0)) \leq  \lim_{j \to \infty} \int_{t\wedge t_0}^{t\vee t_0} \eta_j(r) \,\d r =  \int_{t\wedge t_0}^{t\vee t_0} \eta(r)\, \d r. $$
Hence, we can bound
\begin{align*}
  &d(u(t),u_j(t)) \leq  d(u(t),u(t_0)) + d(u(t_0),u_j(t_0)) +
 d(u_j(t_0),u_j(t))  \\
  &\quad  \leq \int_{t\wedge t_0}^{t\vee t_0} \eta (r)\, \dr +  d(u(t_0),u_j(t_0)) + \int_{t\wedge t_0}^{t\vee t_0} 
     \eta_j (r)\, \d r  \quad \text{for a.e.} \ t \in [0,T].
\end{align*}
We then use that $\eta_j$  converges  to  $\eta$ to further bound  the right-hand side. It is thus bounded independently of $j$ and $t
\in [0,T]$, hence the
Dominated Convergence Theorem implies that $d(u_j, v) \to
d(u, v)$ strongly in  $L^r(0,T)$  for every $r \in (1,\infty)$.   
  Using the lower semicontinuity of $\phi$ and Fatou's Lemma  as
  $ \phi  $ is nonnegative  we can
  hence  deduce that 
  \begin{align*}
    &\EE(\alpha,u) = \int_0^T
    \left(-\alpha(r)d(u(r), v(r))+\phi(u(r)) - \phi( v (r))
    -  \mu_0  d^q(u(r), v (r)) \right)\, \dr \\
    &\leq \liminf_{j\to \infty} \int_0^T
    \left(-\alpha_j(r)d(u_j(r), v(r)){+}\phi(u_j(r)) {-} \phi( v (r))
     {-}  \mu_0  d^q(u_j(r), v (r))  \right)\, \dr \\
    & \leq \liminf_{j\to \infty} \int_0^T  E
    (\alpha_j(r),u_j (r))\, \dr
  \end{align*}
   by taking the maximum. This proves  the claim.
\end{proof}

\begin{corollary}[Lower semicontinuity of $\GG$  and $\HH$]\label{prop:FFlsc}
  Let $(\alpha_j,u_j)_j \in K(u^0) $ be  such that
  $\alpha_j \to \alpha$ weakly in  $L^{p'}(0,T)$,   $u_j(t)
  \to u(t)$ for  almost  all $t\in[0,T]$  with $u(0)=u^0$, $u \in AC^p([0,T];U)$,
  \[
 d(u_j(s), u_j(t)) \le \int_s^t \eta_j(r) \,\dr \quad \forall 0 \le s \le t \le T, j \in \Nz
  \]
for some $\eta_j \in L^1(0,T)$ converging to $\eta$ weakly in $L^1(0,T)$, and $u_j(T) \to u(T)$. Then  
  \begin{equation*}
    \GG(\alpha,u)\leq \liminf_{j\to \infty}
    \GG(\alpha_j,u_j)\quad  \text{and} \quad \HH(\alpha,u)\leq \liminf_{j\to \infty}
    \HH(\alpha_j,u_j).
  \end{equation*} 
\end{corollary}
\begin{proof}
Owing to Proposition \ref{prop:BBlsc} we have
that
$$0\leq \EE (\alpha,u) \leq \liminf_{j \to \infty}\EE(\alpha_j ,
u_j) =0.$$
This ensures that $\sl(u) \in L^{p'}(0,T)$ as $ \sl(u)\leq \alpha$ almost
everywhere in $[0,T]$  by \eqref{eq:max}.
 Hence, $(\alpha,u)\in
K(u^0)$ and we can hence pass to the $\liminf$ in
$\FF(\alpha_j,u_j)$ as $j
\to \infty$ and find
\begin{align}
  &\FF(\alpha,u)  =  \phi(u(T))-\phi(u^0) +   \int_0^T \psi(|u'|(r))\,
    \dr +  \int_0^T \psi^*(\alpha(r))\, \dr \nonumber\\
  &\quad \leq \liminf_{j\to \infty}\FF(\alpha_j,u_j)\label{eq:lscdopo}
\end{align}
whence the  lower semicontinuity of $\GG$  and $\HH$  follows.
\end{proof}

We now   prove that  $
\GG$ and $\HH$ can actually be minimized on  $K(u^0)$.  The existence
of minimizers  obviously requires that  these functionals are  proper on $K(u^0)$, see condition~\eqref{eq:assB1}.  

\begin{proposition}[Minimization of $\GG$ and
  $\HH$] \label{prop:min}  Assume that
  \begin{equation}\label{eq:assB1}
    \inf_{K(u^0)}\GG <\infty.
\end{equation}
Then, $\GG$ admits a minimizer in $K(u^0)$. Any
minimizer $(\alpha,u)$ of $\GG$  in $K(u^0)$
 is such that $\alpha = \sl (u)$ almost everywhere in
 $[0,T]$.
 If condition \eqref{eq:assB1} holds for $\HH$ in place of $\GG$, then $\HH$ admits a minimizer in $K(u^0)$.
\end{proposition}
\begin{proof}
Under
condition \eqref{eq:assB1}, let $(\alpha_j,u_j)_j\in K(u^0)$ be a minimizing sequence. With no loss of generality one can assume that $
\GG(\alpha_j , u_j) \leq  C$  for some  $C \in (0, \infty)$.  This in
particular entails that
\begin{equation}
  \phi(u_j(T)) +  \int_0^T \psi(|u_j'| (r))\, \dr +  \int_0^T
\psi^*(\alpha_j(r))\, \dr \leq  C  +
\phi(u^0). \label{eq:lo1}
  \end{equation}
As $\EE(\alpha_j,u_j)=0$, one has that $u_j \in D(\sl)$ and $\sl (u_j) \leq \alpha_j$
almost everywhere in $[0,T]$. This gives
  for all $t \in [0,T] $ that
  \begin{align*}
    &\phi(u_j(t)) \leq \phi(u^0) + |\phi(u_j(t)) - \phi(u^0)| \leq \phi(u^0) + 
      \int_0^t \sl(u_j(r))\, |u_j'|(r)\, \dr \\
      &\quad\leq \phi(u^0) +  \int_0^T \alpha_j(r)\, |u_j'|(r)\,
      \dr \\
    &\quad \leq \phi(u^0) + \int_0^T \psi(|u_j'|(r))\,\dr + 
      \int_0^T \psi^*(\alpha_j(r))\, \dr \stackrel{\eqref{eq:lo1}}{\leq}
       C   + 2\phi(u^0).
  \end{align*}
   Using \eqref{eq:lo1} and \eqref{eq:A4}, we can apply Lemma \ref{lemma:compactness}(a)  and extract not relabeled
 subsequences such that $u_j (t) \to u(t)$ for
all $t \in [0,T]$, as well as  $\alpha_j \to \alpha$ weakly in
$L^{p'}(0,T)$ and $|u_j'|\to \eta$ weakly in $L^p(0,T)$  for some $u \in AC^p([0,T];U)$ with $|u'| \le \eta$.  Proposition \ref{prop:FFlsc} ensures that
$(\alpha,u)\in K(u^0)$ minimizes $\GG $.
  
  Given any minimizer $(\alpha,u)\in K(u^0)$ of $\GG$, from
  $\EE(\alpha,u)=0$ one has that $u \in D(\sl)$ and $\sl
  (u)\leq \alpha$ almost everywhere in $[0,T]$. This implies that
  $\GG (\sl(u),u) \leq \GG(\alpha,u)$, the inequality being
  strict if $\sl(u)\not =\alpha$ on a nonnegligible subset of $[0,T]$
   due to the strict monotonicity of $\psi$.   In
  particular, $ \sl(u)=\alpha$ almost everywhere in $[0,T]$.

We now turn our attention to $\HH$. Under condition
  \eqref{eq:assB1}, now written for $\HH$ instead of $\GG$,  the bound in \eqref{eq:lo1} still holds and we also have, using the $2^{q-1}$-subadditivity  of  $|\cdot|^q$,  
 \begin{align*}
  &\phi(u_j(T)) +  \int_0^T \psi(|u_j'| (r))\, \dr +  \int_0^T
  \psi^*(\alpha_j(r))\, \dr + \int_0^T \phi(u_j(t))\, \d t \\
  &\quad \leq  C  +\int_0^T \alpha_j(t) d (u_j(t),u_*)\, \d t +
\phi(u^0)  +  \mu_0  \int_0^T d^q(u_j(t), u_*)\, \d t \\
& \quad  \le C+ \frac{1}{p'}\int_0^T|\alpha_j|^{p'}(t)\, \d t+ \frac1p\int_0^T
     d^p(u_j(t),u^0) \, \d t + \phi(u^0)\\
     &\qquad +  { 2^{q-1} \mu_0 } \int_0^T d^q(u_j(t),
       u^0) \, \d t + { 2^{q-1} \mu_0   T} d^q(u_j(t), u_*),
 \end{align*}
 where $u_* \in U$ is such that $\phi(u_*)=0$. Moreover, for every $\sigma>1$, $t \in [0,T]$ and $j \in \Nz$, we have
\[
d^\sigma(u_j(t), u^0) \le \left ( \int_0^t |u_j'|(r) \, \d r \right )^\sigma \le T^{\sigma/p'}\left ( \int_0^T |u_j'|^p(r) \, \d r \right )^{\sigma/p} 
\]
which is uniformly bounded due to \eqref{eq:lo1} and \eqref{eq:A4}. We can thus apply Lemma \ref{lemma:compactness}(b) twice, once with $t_0=0$ and once with $t_0=T$, and extract not relabeled subsequences such that $u_j (t) \to u(t)$ for
a.e. $t \in [0,T]$, $u(0)=u^0$, $u_j(T) \to u(T)$, as well as  $\alpha_j \to \alpha$ weakly in
$L^{p'}(0,T)$ and $|u_j'|\to \eta$ weakly in $L^p(0,T)$ for some $u \in AC^p([0,T];U)$ with $|u'| \le \eta$.  Proposition \ref{prop:FFlsc} ensures that
$(\alpha,u)\in K(u^0)$ minimizes $\HH $.  
\end{proof}

\begin{remark}[Validity of condition \eqref{eq:assB1}]
   Let us now comment on condition \eqref{eq:assB1}. This may of
  course  be met,  as curves of maximal slope $\ove u$ do exist in this setting,  see \cite{cuore3}.  Specifically, we have that
  $\GG (\sl(\ove u),\ove u) = \HH (\sl(\ove u),\ove u)=0$  by Theorem \ref{prop:BEN}  and
   \eqref{eq:assB1} trivially holds true  for both $\GG$ and
  $\HH$. 

  Not relying on the existence of curves of maximal slope,
  \eqref{eq:assB1} also holds in the regular case $u^0 \in D(\sl)$ by
  simply choosing $(\ove \alpha, \ove u)=(\sl (u^0),u^0)$ so that
  $\GG (\sl (u^0),u^0) = \HH(\sl (u^0),u^0) =T
  \psi^*(\sl(u^0))<\infty$. More generally, it suffices to find any
  curve $u \in AC^p([0,T];U)$ with $u(0)=u^0$ and
  $\sl (u)\in L^{p'}(0,T)$. This can be done by resorting to the
  homogeneous case $\psi(t)=|t|^p/p$. Indeed, existence of curves of
  maximal slope for $\psi(t)=|t|^p/p$ under assumptions
  \eqref{eq:A1}--\eqref{eq:A3} is known \cite[Rem.~3.2.5,
  p.~69]{Ambrosio08}.  Letting $\ove u$ be one of such curves, one has
  that $(\sl(\ove u),\ove u)\in K(u^0)$ and
  \begin{align*}
    &\GG (\sl(\ove u),\ove u) = \phi(\ove u(T)) - \phi(u^0) +
      \int_0^T\psi(|\ove u'|(r))\, \dr + \int_0^T
      \psi^*(\sl(\ove u(r)))\, \d r \\
    & \quad \leq \phi(\ove u(T)) - \phi(u^0) +
      \int_0^T C_\psi \big(  |\ove u'|^p(r) + \sl^{p'}(\ove u(r)) +1\big)\, \d
      r <\infty, \\
    & \HH (\sl(\ove u),\ove u) = \left(\phi(\ove u(T)) - \phi(u^0) +
      \int_0^T\psi(|\ove u'|(r))\, \dr + \int_0^T
      \psi^*(\sl(\ove u(r)))\, \d r\right)^+ \\
    & \quad  \leq \phi(\ove u(T)) +
      \int_0^T C_\psi \big(  |\ove u'|^p(r) + \sl^{p'}(\ove u(r)) +1\big)\, \d
      r <\infty.
  \end{align*}
\end{remark}
\begin{remark}[Null-minimization]
   Checking
  $\min_{K(u^0)}\GG(\alpha,u) = \min_{K(u^0)}\HH(\alpha,u) = 0$
  without resorting to the existence of curves of maximal slope 
  is
  more demanding. One is asked to find trajectories $u_\epsi$ which
  are arbitrarily close to curves of maximal slope, in the precise
quantitative sense given by
  $\GG ( \alpha_\epsi, u_\epsi)<\epsi$  or
  $ \HH ( \alpha_\epsi, u_\epsi)<\epsi$, for some $\alpha_\epsi \in
  L^{p'}(0,T)$. We are not in the
  position of presenting a general argument to this
  effect. Nonetheless, we expect that this could be achieved by
  considering some interpolation of discrete
  solutions of a suitable minimizing-movement scheme for time
  partitions of small diameter \cite{cuore3}. 
\end{remark}
 
 We conclude by showing the possibility of finding a curve of maximal
 slope by a penalization approach. 

 \begin{proposition}[Penalization via $\GG_\delta$]\label{prop:pen2}
   For all $\delta>0$, let $(\alpha_\delta,u_\delta)\in K(u^0)$
   minimize $\GG_\delta$ and assume that
   \begin{equation}\label{eq:inf}
    \inf_{K(u^0)} \GG_\delta \to  0  \ \ \text{ as } \ \ \delta \to 0.
\end{equation}
Then, we have that $\alpha_\delta \to \alpha$ weakly in $L^{p'}(0,T)$, $u_\delta
\to u$ almost everywhere in $[0,T]$,  and $|u_\delta'| \to \eta$ weakly in
$L^p(0,T)$, up to a not relabeled subequence, where $u$ is a curve of maximal slope.  
\end{proposition}

\begin{proof}
 From minimality and \eqref{eq:inf}, for all $\epsi >0 $ we have that $\GG_\delta
 (\alpha_\delta,u_\delta)  <\epsi$ for $\delta$ small enough,
  which implies that
  \begin{align}
     0 &\le \FF(\alpha_\delta, u_\delta)-\phi(u^0)  \nonumber\\
    \quad &\leq \phi(u_\delta(T)) + \int_0^T \psi(|u_\delta'|(t))\, \d t +
    \int_0^T\psi^*(\alpha_\delta(t))\, \d t \leq \phi(u^0) +\epsi. \label{eq:pen}
  \end{align}
  At the same time, we have that $ \EE(\alpha_\delta,u_\delta)/\delta=  \GG_\delta
 (\alpha_\delta,u_\delta) -\FF  
 (\alpha_\delta,u_\delta) \leq \epsi +\phi(u^0)$ ensuring that
 \begin{align*}
   & \int_0^T \phi(u_\delta(t))\, \d t \leq \int_0^T \alpha_\delta(t)
     d(u_\delta(t),u_*)\, \d t + \delta \epsi +\delta \phi(u^0) +
     { \mu_0  }\int_0^T d^q(u_j(t), u_*) \, \d t, 
 \end{align*} 
  where $u_* \in U$ is such that $\phi(u_*)=0$. 
 By using the bounds in \eqref{eq:pen} and arguing as in
  the proof of Proposition \ref{prop:min},  we have that the above right-hand side is bounded independently of $\delta$ as
 $\delta \to 0$.  We can thus apply Lemma \ref{lemma:compactness}(b) twice, once with $t_0=0$ and once with $t_0=T$, and extract not relabeled subsequences such that $u_j (t) \to u(t)$ for
a.e. $t \in [0,T]$, $u(0)=u^0$, $u_j(T) \to u(T)$, as well as  $\alpha_j \to \alpha$ weakly in
$L^{p'}(0,T)$ and $|u_j'|\to \eta$ weakly in $L^p(0,T)$ for some $u
\in AC^p([0,T];U)$ with $|u'| \le \eta$  almost everywhere in $[0,T]$. 
 Thus  $\FF(\alpha,u) \leq \liminf_{\delta \to
  0}\FF(\alpha_\delta,u_\delta)$ and  $\EE(\alpha,u) \leq \liminf_{\delta \to
  0}\EE (\alpha_\delta,u_\delta)$. This implies that $$0\leq \EE(\alpha,u)  \leq \liminf_{\delta \to
  0} \delta(\epsi+\phi(u^0)) =0.$$
We hence conclude that $\GG(\alpha,u) =\FF(\alpha,u)  \leq \liminf_{\delta \to
  0} \GG_\delta(\alpha_\delta,u_\delta)<\epsi$ and the assertion
follows by taking $\epsi \to 0$  and using Theorem \ref{prop:BEN}.
\end{proof}

 Before closing this section, let us note that in fact $\GG_\delta \to \GG$ in the $\Gamma$-convergence sense with respect to the topology specified in Corollary \ref{prop:FFlsc}. Indeed, for all $(\alpha_\delta,u_\delta)_\delta$ with $\alpha_\delta\to \alpha $ weakly in $L^{p'}(0,T)$, $u_\delta(t)\to u(t)$ for almost all $t \in [0,T]$ with $u(0)=u^0$, $d(u_\delta(t),u_\delta(s)) \leq \int_s^t \eta_\delta(r)\, \d r $ for all $0\leq s <t \leq T$
with $\eta_\delta \to \eta$ weakly in $L^p(0,T)$, and $u_\delta(T)\to u(T)$, we have that $\FF(\alpha,u)\leq \liminf_{\delta\to 0}\FF(\alpha_\delta,u_\delta)$. By assuming with no loss of generality that $\sup_j\GG_\delta(\alpha_\delta,u_\delta)<\infty$, we readily get that $\EE(\alpha,u)=0$. Hence, $\GG(\alpha,u)\leq\liminf_{\delta\to 0}\GG_\delta(\alpha_\delta,u_\delta)$. On the other hand, $\GG_\delta$ converges to $\GG$ pointwise, which directly entails the existence of recovery sequences.

\section{An extension  
   without assumption  \eqref{eq:A20}}\label{sec:beyond}

This last section is devoted to an extension of the variational
characterization of Proposition \ref{prop:char}, and hence of the
Brezis--Ekeland--Nayroles principle of  Theorem \ref{prop:BEN},
not relying on the validity of  assumption \eqref{eq:A20} concerning the non-asymptotic formulation of the local slope.

For the sake of notational convenience, within this section we still
indicate by $\mu_0 \in [0,\infty)$ the smallest $\mu_0$ such that
$\sl=\sl_{\mu_0}$ if such a value exists, and we set $\mu_0=\infty$ 
otherwise. Let
us recall that $\sl$ is a strong upper gradient for $\phi$ if
$\mu_0<\infty$. In case $\mu_0 =\infty$, this is to be additionally
assumed as it is needed in the proof of the
Brezis--Ekeland--Nayroles principle. Henceforth, in addition to \eqref{eq:A1}--\eqref{eq:A_min0} and
\eqref{eq:A3}--\eqref{eq:A4}, we ask that  
\begin{equation}
  \label{eq:A200}
 \ \sl \ \text{is a strong upper gradient
    for $\phi$}.
\end{equation}
Note again that the latter is weaker than \eqref{eq:A20}, which is not
assumed here.

The first step to extend the results of the previous sections  is to obtain an alternative
variational representation of the local slope $\sl$  covering the
case $\mu_0=\infty$, as well.  
 Let us start by proving the following properties of the
$\mu$-slope $\sl_\mu$ defined in \eqref{eq:slmu}.
\begin{proposition}[Properties of the $\mu$-slope]
  For all  $u\in U$ one has that
  \begin{align}
    &\text{the map} \ \mu\in [0,\infty) \mapsto \sl_\mu(u) \ \text{is nonincreasing},\label{eq:mu0}\\[2mm]
    &\sl(u) \leq \sl_\mu(u)
  \leq  \sl_0(u)=  \gl(u),
      \label{eq:mu1}\\[2mm]
    &\lim_{\mu\to \infty}\sl_\mu(u) = \sl(u), \label{eq:mu2}\\
    &\lim_{\mu\to 0}\sl_\mu(u) =  \sl_0(u)=  \gl(u). \label{eq:mu3}
  \end{align}
In case $\mu_0 <\infty$ we have that $\sl  = \sl_\mu
$ for all $\mu\geq \mu_0$.
\end{proposition}
\begin{proof}
  The monotonicity  \eqref{eq:mu0}  of the map $\mu \mapsto \sl_\mu(u)$ readily follows
  from the definition of the $\mu$-slope.
  
  For all $v\in U$  with $v \ne u$  one has that
  $$\frac{(\phi(u) -
    \phi(v))^+}{d(u,v)} -  {\mu}d^{q-1}(u,v) \leq \sl_\mu(u)
  \leq  \sl_0(u)= \gl(u).$$
  Taking the $\limsup$ as $v\to u$ proves \eqref{eq:mu1}.

  Fix $u \in D(\sl)$.   Since   $\mu:(0,\infty)\mapsto
  \sl_\mu(u)\in [0,\infty]$  is nonincreasing, the  
  limits in \eqref{eq:mu2}--\eqref{eq:mu3} exist. For all $\epsi>0$ we can find $r>0$ such that 
$$\frac{(\phi(u) -
    \phi(v))^+}{d(u,v)} \leq \sl(u)+\epsi \quad \forall v \in B_r\coloneqq\{w\in U :  0 <  d(u,w)<r\} .$$
  Note that $r$ is independent of $\mu$. We hence have that 
  $$\frac{(\phi(u) -
    \phi(v))^+}{d(u,v)} -  {\mu}d^{q-1}(u,v) \leq \sl(u)+\epsi
  \quad \forall v \in B_r .$$
  On the other hand
   \begin{align*}
       & \frac{(\phi(u) -
    \phi(v))^+}{d(u,v)} -  {\mu}d^{q-1}(u,v) \leq \frac{(\phi(u) -
    \phi(v))^+}{r} -  {\mu}r^{q-1}  \\
    &\qquad \le \frac{\phi(u)}{r} - \mu r^{q-1}  \quad \forall v \in
  U\setminus B_r.
  \end{align*}
  As $\lim_{\mu \to \infty}(-\mu r^{q-1}) = -\infty$,  for $\mu$ large enough,  the supremum of the
  above left-hand side is taken in $B_r$. This in particular proves that  
  $ \lim_{\mu \to \infty} \sl_\mu(u) \leq \sl(u)+\epsi$ 
  and \eqref{eq:mu2} follows as $\epsi$ is arbitrary.
  
 Let $u \in U$ and $(v_k)_k \subset U$ with $v_k \ne u$ be such that
\[
\gl(u) = \lim_{k \to \infty} \frac{(\phi(u) -
    \phi(v_k))^+}{d(u,v_k)}.
\]
Thus, for every $\mu \in (0,\infty)$, we have
\begin{align*}
\frac{(\phi(u) -  \phi(v_k))^+}{d(u,v_k)} &= \frac{(\phi(u) -  \phi(v_k))^+}{d(u,v_k)} - \mu d^{q-1}(u,v_k) + \mu d^{q-1} (u,v_k)
  \\
  &\ \leq \sl_\mu(u) + \mu d^{q-1} (u,v_k).
\end{align*}
Passing first to the limit as $\mu \to 0$ and then as $k \to \infty$, we get $\gl (u) \le \lim_{\mu \to 0} \sl_{\mu}(u)$.

  In case $\mu_0 <\infty$ one has that $\sl_{\mu_0}(u) = \sl(u)$. As
  $\mu \mapsto \sl_\mu(u)$ is nonincreasing by \eqref{eq:mu0}, from
  \eqref{eq:mu1} we get
  $\sl(u) \leq \sl_\mu(u) \leq \sl_{\mu_0}(u) = \sl(u)$ for all $\mu
  \geq \mu_0$. 
\end{proof}

 We now  introduce a modification
$\tilde E$ of the functional $E$, for which the  analogue  of \eqref{eq:max} holds. 
Arguing as in Lemma \ref{lemma:v}, we define the functionals $\ove E_{ \mu}: [0,\infty) \times U \times
D(\phi) \to (-\infty,\infty]$ and $E_\mu: [0,\infty) \times U \to
[0,\infty]$ as
\begin{align*}
  &\ove E_\mu(\alpha,u,v) = -\alpha d(u,v) + \phi(u)
  -\phi(v)- {\mu} d^q(u,v),\\
  &E_\mu(\alpha,u) = \max_{v\in D(\phi)} \ove E_\mu(\alpha,u,v).
\end{align*}
Note  that, if $u \in D(\phi)$  both $\alpha\in [0,\infty)  \mapsto E_\mu(\alpha,u)$ and
$\mu \in [0,\infty)  \mapsto  E_\mu(\alpha,u)$ are
nonincreasing, convex  and finite,  hence
continuous on  $(0,\infty)$.  Both functions are instead identically equal to $\infty$ if $u \notin D(\phi)$. 
Proposition \ref{prop:char}, now applied to $E_\mu$, ensures that
\begin{equation}\sl_\mu(u) \leq \alpha \ \Leftrightarrow \
E_\mu(\alpha,u)=0.\label{eq:emu}
\end{equation}

Set $h:[0,\infty) \to [0,\infty]$ to be lower semicontinuous and
nondecreasing, with $h(0)=0$ and $h(0+)>0$, a possible choice being the 
indicator
function of $0$, namely $h(0)=0$ and $h(r)=\infty$ for all
$r>0$. Using such $h$ we define
$\tilde E :[0,\infty) \times U \to
[0,\infty]$ as
\begin{equation}\label{eq:nuovo}\tilde E(\alpha,u) = \sup_{\epsi \in \Qz_+}
  \inf_{ \mu_0\geq \mu \in \Qz_+} h(E_\mu(\alpha+\epsi,u))
\end{equation}
where $\Qz_+ \coloneqq \Qz \cap (0,\infty)$.
We have the following variational characterization of the local slope.

 \begin{proposition}[Characterization 2]\label{prop:char2}  For every $(\alpha, u) \in [0,\infty) \times U$, we  have
   that
   \begin{equation}\sl(u) \leq \alpha \ \Leftrightarrow  \ \tilde
   E(\alpha,u)=0.\label{eq:char2}
   \end{equation}
 \end{proposition}
 \begin{proof}
    The case $\mu_0 <\infty$  can be checked by easily adapting
   the proof of Proposition~\ref{prop:char}.

   Set $\mu_0 = \infty$. Assume $\sl(u)\leq \alpha$
   and fix any $\epsi \in \Qz_+$. As $\sl_\mu(u) \to \sl(u)$ from
   \eqref{eq:mu2}, there exists $\mu\in \Qz_+$ such that $\sl_\mu(u) \leq
   \alpha+\epsi$. Then, \eqref{eq:emu} implies that
   $E_\mu(\alpha+\epsi,u) =0$, so that $\inf_{ \mu \in
     \Qz_+} h(E_\mu(\alpha+\epsi,u))  =0$ and $\tilde E(\alpha,u)=0$ follows
   as $\epsi $ is arbitrary.

   Assume  now that $\tilde E(\alpha,u)=0$. Then, for all $\epsi\in \Qz_+$
   there exists  $\mu\in \Qz_+$ such that $E_\mu(\alpha+\epsi,u)
   =0$. This implies that
   $\sl_\mu(u) \leq
   \alpha+\epsi$ by \eqref{eq:emu}. As $\sl(u) \leq \sl_\mu(u)$ by
   \eqref{eq:mu1}, we deduce that $\sl(u) \leq \alpha$ by taking
   $\epsi \to 0$.
 \end{proof}

Let us now introduce the time-continuous functionals, based on 
 the  functional $\tilde E$.
Note that for all $\mu,\, \epsi>0$ the map $(\alpha,u)\mapsto
E_\mu(\alpha+\epsi,u)$ is Borel, which implies that $\tilde E$ is Borel, as
well. This allows to define the functionals $\tilde \EE, \, \tilde \GG, \, \tilde \HH : L^p(0,T) \times
AC^p([0,T];U) \to [0,\infty]$ as
\begin{align*}
  &\tilde \EE(\alpha,u)=\int_0^T \tilde E(\alpha(t),u(t))\, \d t, \\
   & \tilde \GG(\alpha,u) =
     \left\{
     \begin{array}{ll}
       \FF(\alpha,u) \quad&\text{if} \ \tilde\EE (\alpha,u)=0,\\[2mm]
       \infty&\text{otherwise},
     \end{array}
               \right. \\[1mm]
  &\tilde \HH(\alpha,u) = (\FF(\alpha,u))^+ + \tilde \EE(\alpha,u).
\end{align*}

Using these functionals, we can again state a
Brezis--Ekeland--Nayroles principle as in Theorem \ref{prop:BEN}.

 \begin{proposition}[Null-minimization 2]\label{prop:BEN2}
    The following are equivalent:
   \begin{itemize}
   \item[\rm (i)]\quad $u$ is a curve of maximal slope,\vspace{0.6mm}
     \item[\rm (ii)]\quad  $
  \tilde \GG(\alpha,u)=\min_{K(u^0)} \tilde \GG=0$ for some $\alpha \in
  L^{p'}(0,T)$,
  \item[\rm (iii)] \quad  $
  \tilde \HH(\alpha,u)=\min_{K(u^0)} \tilde \HH=0$ for some $\alpha \in
  L^{p'}(0,T)$.
   \end{itemize}  
 \end{proposition}

 We omit the proof of Proposition  \ref{prop:BEN2},  as it coincides with  that  of Theorem
 \ref{prop:BEN}, up to using the characterization \eqref{eq:char2} instead of
 \eqref{eq:max}.

Before closing this section, let us point out that, in case
$\mu_0<\infty$, the above functionals  play the same role as  the ones of
the previous sections. More  precisely, if $\mu_0<\infty$, one has that the functionals $\tilde E$ and $E$ share the same sublevel
$0$, which in particular implies that $\tilde \GG = \GG$. This is an immediate consequence of
Proposition  \ref{prop:char2} and  of the characterization \eqref{eq:max}.  

\begin{proposition}[Case $\mu_0 <\infty$]
   If $\mu_0 <\infty$  one has
  that
  $$\tilde E(\alpha,u)  = 0 \ \Leftrightarrow \  E(\alpha,u)=0.$$
\end{proposition}

\textbf{Acknowledgement.}
 This research is partially supported by the Austrian Science Fund (FWF) project
 10.55776/PAT1408325.
 For open access purposes, the authors have applied a CC BY public
copyright license to any author accepted manuscript version arising
from this submission.

\textbf{Competing Interests.}
The authors have no conflicts of interest to declare that are relevant to the content of this chapter.

\textbf{Tool and computational resource disclosure.}
The authors used a LLM as a language support tool. Moreover, the proof of the limit \eqref{eq:mu2} has been inspired by a LLM comment.



\begin{thebibliography}{99}

\bibitem{Ambrosio95}
 L.~Ambrosio. Minimizing movements. {\it Rend. Accad. Naz. Sci. XL Mem. Mat. Appl. (5)}, 19 (1995), 191--246.

 \bibitem{Ambrosio08}
 L.~Ambrosio, N.~Gigli, G.~Savar{\'e}. \emph{Gradient flows in
   metric spaces and in the space of probability measures}, second ed., 
 	Birkh\"auser Verlag, Basel, 2008.
   
\bibitem{cuore3}
P.-C.~Aubin--Frankowski, G.~E.~Sodini, U.~Stefanelli. The
Brezis--Ekeland--Nayroles principle in metric spaces:
minimizing-movement schemes. In preparation (2026).
 
\bibitem{Barbu12}
V.~Barbu. Optimal control approach to nonlinear diffusion equations
driven by Wiener noise. {\it J. Optim. Theory Appl.} 153 (2012),
no. 1, 1--26.

\bibitem{Barbu-Kunisch95}
 V.~Barbu, K.~Kunisch.  Identification of nonlinear parabolic
  equations. {\it Control Theory Adv. Tech.} 10 (1995), no.~4, part~5, 1959--1980.

\bibitem{Barbu}
V.~Barbu, M.~R\"ockner. Variational solutions to nonlinear stochastic
differential equations in Hilbert spaces. {\it Stoch. Partial
  Differ. Equ. Anal. Comput.} 6 (2018), no. 3, 500--524.

\bibitem{Biot55}
M.~A. Biot. {Variational principles in irreversible thermodynamics
  with application to viscoelasticity}. {\it Phys. Rev. (2)}, 97
(1955), 1463--1469.

\bibitem{Bog07}
V.~I. Bogachev, 
\newblock {\em Measure theory. Vol. I, II}.
\newblock Springer, Berlin, 2007.

\bibitem{Boroushaki}
 S.~Boroushaki, N.~ Ghoussoub. A self-dual variational approach to
 stochastic partial differential equations. {\it J. Funct. Anal.} 276
 (2019), no. 4, 1201--1243.

\bibitem{Brezis-Ekeland76b}
H.~Brezis, I.~Ekeland.
\newblock Un principe variationnel associ\'e \`a certaines \'equations
  paraboliques. {L}e cas d\'ependant du temps.
\newblock {\em C. R. Acad. Sci. Paris S\'er. A-B}, 282 (1976), 
no.~20, Ai,  A1197--A1198.

\bibitem{Brezis-Ekeland76}
H.~Brezis, I.~Ekeland.
\newblock Un principe variationnel associ\'e \`a certaines \'equations
  paraboliques. {L}e cas ind\'ependant du temps.
\newblock {\em C. R. Acad. Sci. Paris S\'er. A-B}, 282 (1976),  
no.~17, Aii,  A971--A974.

\bibitem{Buliga}
M.~Buliga, G.~de Saxc\'e. A symplectic Brezis-Ekeland-Nayroles
principle. {\it  Math. Mech. Solids}, 22 (2017), no. 6, 1288--1302.

\bibitem{Cao}
X.~Cao, A.~Oueslati, N.~Shirafkan, F.~Bamer, B.~Markert, G.~de
Saxc\'e. A non-incremental numerical method for dynamic elastoplastic
problems by the symplectic Brezis--Ekeland--Nayroles principle. {\it Comput. Methods Appl. Mech. Engrg.} 384 (2021), 113908, 25 pp.

\bibitem{Carini}
L.~Carini, M.~Jensen, R.~N\"urnberg. Deep learning for gradient flows
using the Brezis-Ekeland principle. {\it Arch. Math. (Brno)}, 59
(2023), no. 3, 249--261.

\bibitem{Cheeger99}
J.~Cheeger.
\newblock Differentiability of {L}ipschitz functions on metric measure spaces.
\newblock {\em Geom. Funct. Anal.} 9 (1999),  no.~4,  428--517.

\bibitem{DeGiorgi80}
E.~De Giorgi, A.~Marino, M.~Tosques.
\newblock Problems of evolution in metric spaces and maximal decreasing curve.
\newblock {\em Atti Accad. Naz. Lincei Rend. Cl. Sci. Fis. Mat. Natur. (8)},
  68 (1980),  no.~3,  180--187.

\bibitem{Dondl19}
  P.~Dondl, T.~Frenzel, A.~Mielke. A gradient system with a wiggly
  energy and relaxed EDP-convergence. {\it ESAIM Control
    Optim. Calc. Var.} 25 (2019), Paper No. 68, 45 pp.

\bibitem{Frenzel}
 T.~Frenzel, M.~Liero. Effective diffusion in thin structures via
 generalized gradient systems and EDP-convergence. {\it Discrete
   Contin. Dyn. Syst. Ser. S}, 14 (2021), no.~1, 395--425.

\bibitem{Ghoussoub08}
N.~Ghoussoub. {\it Self-dual partial differential systems and their variational principles}. Springer Monographs in Mathematics. Springer, New York, 2009. 
 
\bibitem{Ghoussoub-McCann04}
{N.~Ghoussoub, R.~J. McCann}. {A least action principle for steepest
  descent in a non-convex landscape}. In {\it Partial differential equations and
  inverse problems}, vol.~362 of {\it Contemp. Math.},
Amer. Math. Soc., Providence, RI, 2004, 177--187.

\bibitem{Ghoussoub2}
N.~Ghoussoub, A.~Moameni. Selfdual variational principles for periodic
solutions of Hamiltonian and other dynamical systems. {\it
  Comm. Partial Differential Equations}, 32 (2007), no. 4-6,
771--795.

\bibitem{Ghoussoub}
N.~Ghoussoub, A.~Moameni. Hamiltonian systems of PDEs with selfdual
boundary conditions. {\it Calc. Var. Partial Differential Equations},
36 (2009), no. 1, 85--118.

\bibitem{Ghoussoub-Tzou04}
{N.~Ghoussoub, L.~Tzou}. {A variational principle for gradient
  flows}. {\it Math. Ann.} 330 (2004), no.~3, 519--549.

\bibitem{Ghoussoub3}
N.~Ghoussoub, L.~Tzou. Iterations of anti-selfdual Lagrangians and
applications to Hamiltonian systems and multiparameter gradient
flows. {\it Calc. Var. Partial Differential Equations}, 26 (2006),
no. 4, 511–534.

\bibitem{Gurtin63}
{M.~E. Gurtin}. {Variational principles in the linear theory of
  viscoelasticity}. {\it Arch. Rational Mech. Anal.} 13 (1963), 179--191.

\bibitem{Gurtin64b}
{M.~E. Gurtin}. {Variational
  principles for linear elastodynamics}. {\it  Arch. Rational Mech. Anal.} 16 (1964),
  34--50.

\bibitem{Hlavacek69}
{I.~Hlav{\'a}{\v{c}}ek}. {Variational principles for parabolic
  equations}. {\it Apl. Mat.} 14 (1969), 278--297.

\bibitem{JKO}
 R.~Jordan, D.~Kinderlehrer, and F.~Otto. \newblock The variational formulation of the Fokker-Planck equation.
 \newblock  {\it SIAM J. Math. Anal.}   29 (1998),  no.~1,   1--17.

\bibitem{Lemaire96}
{B.~Lemaire}.  An asymptotical variational principle associated with
  the steepest descent method for a convex function. {\it J. Convex Anal.} 3
  (1996), no.~1, 63--70.

\bibitem{Mabrouk01}
M.~Mabrouk. Un principe
  variationnel pour une \'equation non lin\'eaire du second ordre en
  temps. {\it C.
  R. Acad. Sci. Paris S\'er. I Math.} 332 (2001), no.~4,  381--386.

\bibitem{Mabrouk03}
M.~Mabrouk. A variational
  principle for a nonlinear differential equation of second order. {\it Adv. in
  Appl. Math.} 31 (2003), no.~2, 388-419.

\bibitem{Marinoschi}
G.~Marinoschi. A duality approach to nonlinear diffusion
equations. {\it Set-Valued Var. Anal.} 22 (2014), no. 4, 783--807.

\bibitem{Mielke20}
  A.~Mielke, A.~Stephan. Coarse-graining via EDP-convergence for
  linear fast-slow reaction systems. {\it Math. Models 
  Meth. Appl. Sci.} 30  (2020), no.~9, 1765--1807.
 
\bibitem{Mielke21}
A.~Mielke, A.~Montefusco, M.~Peletier. Exploring families of
energy-dissipation landscapes via tilting: three types of EDP
convergence. {\it Contin. Mech. Thermodyn.} 33 (2021), no.~3,
611--637.

\bibitem{Muratori-Savare}
M.~Muratori, G.~ Savar\'e. Gradient flows and evolution variational
inequalities in metric spaces. I: Structural properties. {\it
  J. Funct. Anal.} 278 (2020),  no.~4,  108347, 67 pp.

\bibitem{Nayroles76}
B.~Nayroles.
\newblock Deux th\'eor\`emes de minimum pour certains syst\`emes dissipatifs.
\newblock {\em C. R. Acad. Sci. Paris S\'er. A-B}, 282 (1976), 
no.~17, Aiv,  A1035--A1038.

\bibitem{Nayroles76b}
B.~Nayroles.
\newblock Un th\'eor\`eme de minimum pour certains syst\`emes dissipatifs.
  {V}ariante hilbertienne.
\newblock {\em Travaux S\'em. Anal. Convexe}, 6 (1976), Exp. 
no.~2, 22 pp. 

\bibitem{Ohta-Zhao}
S.~  Ohta, W.~ Zhao. Gradient flows in asymmetric metric spaces and
applications. {\it Ann. Sc. Norm. Super. Pisa Cl. Sci. (5)}, 26 (2025), no. 2, 919--991.

\bibitem{Otto01}
F.~Otto. \newblock The geometry of dissipative evolution equations:
the porous medium equation.  \newblock {\it Comm. Partial
  Differential Equations},  26 (2001),  no.~1-2,   101--174.

\bibitem{Poliakovsky2}
 A.~Poliakovsky. On a variational approach to the method of
 vanishing viscosity for conservation laws. {\it Adv. Math. Sci. Appl.} 18
 (2008), no. 2, 429--451.

\bibitem{Poliakovsky}
 A.~Poliakovsky. Variational resolution for some general classes of
 nonlinear evolutions. Part I. {\it Asymptot. Anal.} 85 (2013),
 no. 1--2, 29--74.

\bibitem{portinale}
 L. Portinale, U. Stefanelli.
 Penalization via global functionals of optimal-control problems for dissipative evolution.
 {\it Adv. Math. Sci. Appl.} 28 (2019), no.~2, 425--447.

\bibitem{Rossi-Savare}
 R.~Rossi, G.~Savar\'e.
Tightness, integral equicontinuity and compactness for evolution
problems in Banach spaces. 
{\it Ann. Sc. Norm. Super. Pisa Cl. Sci. (5)}, 2 (2003), no. 2, 395--431. 

\bibitem{Rossi11}
 R.~Rossi, G.~Savar\'e, A.~Segatti, U.~Stefanelli.  A variational principle for gradient flows
in metric spaces. {\it C. R. Math. Acad. Sci. Paris}, 349 (2011), no. 23--24,
1225--1228.

\bibitem{Rossi19}
 R.~Rossi, G.~Savar\'e, A.~Segatti, U.~Stefanelli.  Weighted
 energy-dissipation principle for gradient flows in metric
 spaces. {\it J. Math. Pures Appl. (9)}, 127 (2019), 1--66. 

\bibitem{rsss}
R.~Rossi, A.~Segatti, U.~Stefanelli.
Global attractors for gradient flows in metric spaces.
{\it J. Math. Pures Appl.} 95 (2011), 204--244.

\bibitem{Sandier04}
E.~Sandier, S.~Serfaty. Gamma-convergence of gradient flows with
  applications to {G}inzburg-{L}andau. {\it  Comm. Pure Appl. Math.} 57 (2004), no.~12,
  1627--1672.
  
\bibitem{Serfaty11}
S.~Serfaty. Gamma-convergence of gradient flows on Hilbert and metric
spaces and applications. {\it Discrete Contin. Dyn. Syst.} 31 (2011), no.~4,
1427--1451.

\bibitem{Shimoyama}
  S.~Shimoyama. Transformation of $p$-gradient flows to $p'$-gradient
  flows in metric spaces. {\it Anal. Geom. Metr. Spaces}, 13 (2025), no. 1,
  Paper No. 20250032, 24 pp.

\bibitem{be}
  U.~Stefanelli. The Brezis--Ekeland principle for doubly nonlinear
  equations. {\it SIAM J. Control Optim.} 47 (2008), no. 3,
  1615--1642.
 
\bibitem{plas}
  U.~Stefanelli.  A variational principle for hardening
  elastoplasticity. {\it SIAM J. Math. Anal.} 40 (2008), no. 2, 623--652.
  
\bibitem{be2}
U.~Stefanelli. The discrete Brezis--Ekeland principle. {\it J. Convex Anal.} 16 (2009), no. 1, 71--87. 

  \bibitem{Stefanelli25}
U.~Stefanelli. The weighted inertia-energy-dissipation principle. {\it
  Math. Models Methods Appl. Sci.} 35 (2025), no. 2, 223--282.

\bibitem{Visintin01}
{A.~Visintin}. {A new approach to
  evolution}. {\it C. R. Acad. Sci. Paris S\'er. I Math.} 332 (2001), no.~3, 233--238.

\bibitem{Visintin11}
A.~ Visintin. Structural stability of doubly-nonlinear flows. {\it Boll. Unione Mat. Ital. (9)}, 4 (2011), no. 3, 363--391.

\bibitem{Visintin13}
A.~Visintin. Variational formulation and structural stability
of monotone equations. {\it Calc. Var. Partial Differential
  Equations}, 47 (2013), no. 1-2, 273--317.

\bibitem{Visintin13b}
A.~Visintin.  Structural stability of rate-independent nonpotential
flows. {\it Discrete Contin. Dyn. Syst. Ser. S}, 6 (2013), no. 1, 257--275.

\bibitem{Visintin16}
A.~Visintin. On Fitzpatrick's theory and stability of flows. {\it Atti
  Accad. Naz. Lincei Rend. Lincei Mat. Appl.} 27 (2016), no. 2,
151--180.

\bibitem{Visintin21}
A.~Visintin. Structural compactness and stability of doubly nonlinear
flows. {\it Pure Appl. Funct. Anal.} 6 (2021), no. 2, 417--432.

\bibitem{Visintin22}
 A.~Visintin. $\Gamma$-compactness and $\Gamma$-stability of maximal monotone
  flows. {\it J. Math. Anal. Appl.} 506 (2022), no. 1, Paper No. 125602, 29
  pp.

\end{thebibliography}
\end{document}